\documentclass[a4paper,12pt]{article}
\usepackage[english]{babel}
\usepackage{graphicx}
\usepackage{amssymb}
\usepackage{amsmath}
\usepackage{amsthm}
\usepackage{latexsym}
\usepackage{enumerate}
\usepackage{geometry}
\usepackage{subcaption}
\newtheorem{theorem}{Theorem}[section]
\newtheorem{lemma}{Lemma}[section]
\newtheorem{definition}{Definition}[section]
\newtheorem{corollary}{Corollary}[section]
\newtheorem{example}{Example}[section]
\title{Spectra of T-vertex and T-edge neighbourhood corona of Two Graphs}

\author{Indranil Mukherjee \thanks{ Ramakrishna Mission Vidyamandira, Howrah, W.B., India. Email:indranil87m@gmail.com} \and Suvra Kanti Chakraborty \thanks{(Corresponding Author) Ramakrishna Mission Vidyamandira, Howrah, W.B., India. Email: suvrakanti@vidyamandira.ac.in} \and  Arpita Das  \thanks{Bethuadahari College, Nadia, W.B., India. Email: arpita.das1201@gmail.com}  }

\date{}

\begin{document}

\maketitle
\begin{abstract}
The \textit{$T$-graph} $T(G)$ of a graph $G$ is the graph whose vertices are the vertices and edges of $G$, with two vertices of $T(G)$ are adjacent if and only if the corresponding elements of $G$ are adjacent or incident. In this paper, we determine the adjacency and Laplacian spectra of $T$-vertex neighborhood corona and $T$-edge neighborhood corona of a connected regular graph with an arbitrary regular graph in terms of their eigenvalues. Moreover, applying these results we construct some non-regular $A$-cospectral and $L$-cospectral graphs.
\end{abstract}

\section{Introduction}
In recent years, construction of cospectral graphs for different matrices is one of the interesting research problem in the area of spectral graph theory. 
All graphs considered in this paper are simple and undirected. Let $G=(V(G),E(G))$ be a graph with vertex set $V(G)$ and edge set $E(G)$. The \textit{adjacency matrix} of $G$, denoted by $A(G)$, is an $n\times n$ symmetric matrix such that $A(u,v)=1$ if and only if vertex $u$ is adjacent to vertex $v$ and $0$ otherwise. If $D(G)$ is the diagonal matrix of vertex degrees of $G$, then the \textit{Laplacian matrix} $L(G)$ is defined as $L(G)=D(G)-A(G)$. For a given matrix $M$ of size $n$, we denote the characteristic polynomial $\det(xI_{n}-M)$ of $M$ by $f_{M}(x)$. The eigenvalues of $A(G)$ and $L(G)$ are denoted by $\lambda_{1}(G)\geq\lambda_{2}(G)\geq\cdots\geq\lambda_{n}(G)$ and $0=\mu_{1}(G)\leq\mu_{2}(G)\leq\cdots\leq\mu_{n}(G)$ respectively and the multiset of these eigenvalues is called as adjacency spectrum and Laplacian spectrum respectively. Two graphs are said to be $A$-cospectral and $L$-cospectral if they have the same $A$-spectrum and $L$-spectrum respectively. Many research works already have done on different kinds of graph operations. One of this is corona operation. For two graphs $G_{1}$ and $G_{2}$ on disjoint sets of $n$ and $m$ vertices, respectively, the corona \cite{Har} $G_{1}\circ G_{2}$ of $G_{1}$ and $G_{2}$ is defined as the graph obtained by taking one copy of $G_{1}$ and $n$ copies of $G_{2}$, and then joining the $i^th$ vertex of $G_{1}$ to every vertex in the $i^th$ copy of $G_{2}$.
The \textit{$T$-graph} $T(G)$ \cite{Cv} of a graph $G$ is the graph whose vertices are the vertices and edges of $G$, with two vertices of $T(G)$ are adjacent if and only if the corresponding elements of $G$ are adjacent or incident. The set of such new vertices corresponding to each edge of $G$ is denoted by $I(G)$ i.e $I(G)=V(T(G))\backslash V(G)$.
In this paper we find the adjacency and Laplacian spectrum of graphs obtained by some corona operations on $T$-graphs, which are defined below.
\begin{definition}
Let $G_{1}$ and $G_{2}$ be two vertex-disjoint graphs with number of vertices $n_{1}$ and $n_{2}$, and edges $m_{1}$ and $m_{2}$, respectively. Then
\begin{enumerate}[(i)]
\item The \textit{$T$-vertex neighbourhood corona} of $G_{1}$ and $G_{2}$, denoted by $G_{1}\boxdot_{T} G_{2}$, is the graph obtained from vertex disjoint union of $T(G_{1})$ and $|V(G_{1})|$ copies of $G_{2}$, and by joining the neighbors of the $i^{th}$ vertex of $V(G_{1})$ to every vertex in the $i^{th}$ copy of $G_{2}$. The graph $G_{1}\boxdot_{T} G_{2}$ has $n_{1}(1+n_{2})+m_{1}$ vertices.
\item The \textit{$T$-edge neighbourhood corona} of $G_{1}$ and $G_{2}$, denoted by $G_{1}\boxminus_{T} G_{2}$, is the graph obtained from vertex disjoint union of $T(G_{1})$ and $|I(G_{1})|$ copies of $G_{2}$, and by joining the neighbors of the $i^{th}$ vertex of $I(G_{1})$ to every vertex in the $i^{th}$ copy of $G_{2}$. The graph $G_{1}\boxminus_{T} G_{2}$ has $m_{1}(1+n_{2})+n_{1}$ vertices.
\end{enumerate}
\end{definition}
\begin{example}
Let us consider two graphs $G_{1}=P_{3}$ and $G_{2}=P_{2}$. The $T$-vertex neighbourhood corona and $T$-edge neighbourhood corona of $G_{1}$ and $G_{2}$ are given in Figure \ref{f1}.
\end{example}\label{f1}
\begin{figure}[ht]
  \subcaptionbox*{}[.4\linewidth]{
    \includegraphics[width=1.2\linewidth]{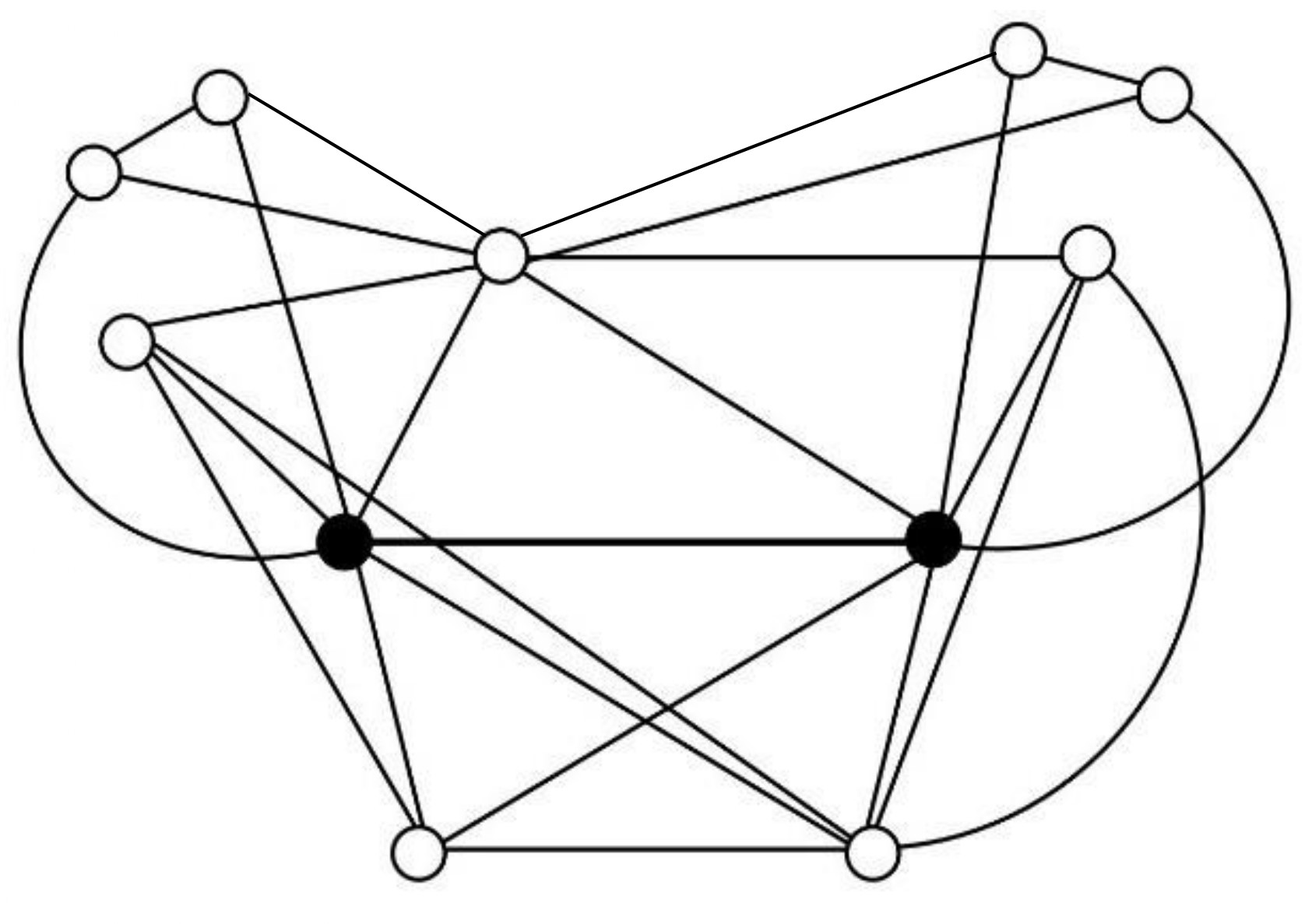}
  }
  \hfill
  \subcaptionbox*{}[.45\linewidth]{
    \includegraphics[width=1\linewidth]{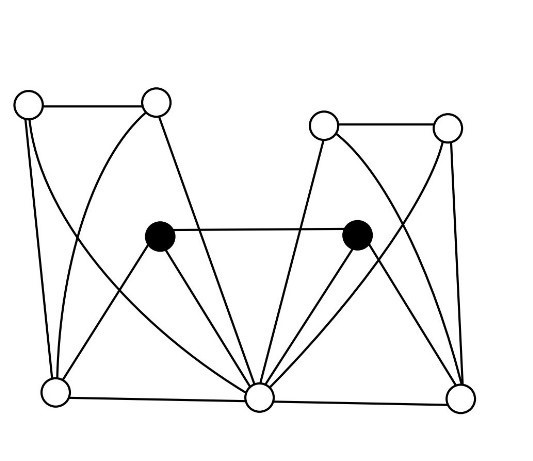}
  }
  \caption{$T$-vertex and $T$-edge neighbourhood corona of $G_1$ and $G_2$}
\end{figure}

Lu and Miao \cite{Lu} determined the adjacency, Laplacian and signless Laplacian spectra of subdivision vertex and edge corona for a regular graph and an arbitray graph in terms of their corresponding spectra. In \cite{Liu}, Liu and Lu found the adjacency, Laplacian and signless Laplacian spectra of subdivision vertex and edge neighbourhood corona of two graphs. Lan and Zhou \cite{Lan} determined the adjacency and Laplacian spectrum of different types of R-coronas for two graphs.
In \cite{Liu2}, Liu et al. determined the resistance distance and Kirchhoff index of $G_{1}\odot_{Q} G_{2}$ and $G_{1}\circleddash_{Q} G_{2}$ of a regular graph $G_{1}$ and an arbitrary graph $G_{2}$. Motivated by these works, here we determine the adjacency and Laplacian spectrum of $G_{1}\boxdot_{T} G_{2}$ and $G_{1}\boxminus_{T} G_{2}$ for a connected regular graph $G_{1}$ and an arbitrary regular graph $G_{2}$ in terms of the corresponding eigenvalues of $G_{1}$ and $G_{2}$. Moreover, applying these results we construct non-regular cospectral graphs.\\
To prove our results we need the following matrix products and few results on them. Recall that the \textit{Kronecker product} of matrices $A=(a_{ij})$ of size $m\times n$ and $B$ of size $p\times q$, denoted by $A\otimes B$, is defined to be the $mp\times nq$ partitioned matrix $(a_{ij}B)$. It is known \cite{Hor} that for matrices $M$, $N$, $P$ and $Q$ of suitable sizes, $MN\otimes PQ=(M\otimes P)(N\otimes Q)$. This implies that for nonsingular matrices $M$ and $N$, $(M\otimes N)^{-1}=M^{-1}\otimes N^{-1}$. It is also known \cite{Hor} that, for square matrices $M$ and $N$ of order $k$ and $s$ respectively, $\det(M\otimes N)=(detM)^{s}(detN)^{k}$.\\
We also need the result given in Lemma \ref{schur} below.
\begin{lemma}(Schur Complement \cite{Cve})\label{schur} Suppose that the order of all four matrices $M$, $N$, $P$ and $Q$ satisfy the rules of operations on matrices. Then we have,
\begin{eqnarray*}
\begin{vmatrix}
  M & N \\
  P & Q
  \end{vmatrix}
&=& |Q||M-NQ^{-1}P|, \mbox{if $Q$ is a non-singular square matrix},\\
&=& |M||Q-PM^{-1}N|, \mbox{if $M$ is a non-singular square matrix}.
\end{eqnarray*}
\end{lemma}
For a graph $G$ with $n$ vertices and $m$ edges, the \emph{vertex-edge incidence matrix} $R(G)$ \cite{God} is a matrix of order $n\times m$, with entry $r_{ij}=1$ if the $i^{th}$ vertex is incident to the $j^{th}$ edge, and $0$ otherwise. It is well known \cite{Cve} that $R(G)R(G)^{T}=A(G)+rI_{n}$ and $A(G)=rI_{n}-L(G)$. So we get that $R(G)R(G)^{T}=2rI_{n}-L(G)$.\\
The \emph{line graph} \cite{God} of a graph $G$ is the graph $\mathcal{L}(G)$, whose vertices are the edges of $G$ and two vertices of $\mathcal{L}(G)$ are adjacent if and only if they are incident on a common vertex in $G$. It is well known \cite{Cve} that $R(G)^{T}R(G)=A(\mathcal{L}(G))+2I_{m}$.
\begin{lemma}\label{line1}
\cite{Cve} Let $G$ be an $r$-regular graph. Then the eigenvalues of $A(\mathcal{L}(G))$ are the eigenvalues of $A(G)+(r-2)I_{n}$ and $-2$ repeated $m-n$ times.
\end{lemma}
If $G$ is an $r$-regular graph, then obviously $L(G)=rI_{n}-A(G)$. Therefore, by Lemma \ref{line1}, we have the following.
\begin{lemma}\label{line2}
For an $r$-regular graph $G$, the eigenvalues of $A(\mathcal{L}(G))$ are the eigenvalues of $2(r-1)I_{n}-L(G)$ and $-2$ repeated $m-n$ times.
\end{lemma}
\section{Our Results}\label{sec:2}
Throughout the paper for any integer $k$, $I_{k}$ denotes the identity matrix of size $k$, $\mathbf{1}_{k}$ denotes the column vector of size $k$ whose all entries are $1$ and $O_{k}$ denotes the zero matrix of size $k$. 
\begin{definition}
\rm{\cite{Cu,Mc} The \textit{$M$-coronal} $\Gamma_{M}(x)$ of an $n\times n$ matrix $M$ is defined as the sum of the entries
of the matrix $(xI_{n}-M)^{-1}$(if exists), that is,\\
$~~~~~~~~~~~~~~~~~~~~~~~~~~~~~~~~~~~~~~\Gamma_{M}(x)=\mathbf{1}^{T}_{n}(xI_{n}-M)^{-1}\mathbf{1}_{n}$}
\end{definition}
\noindent
The following Lemma is straightforward.
\begin{lemma}\label{le2}
\cite{Cu} If $M$ is an $n\times n$ matrix with each row sum equal to a constant
$t$, then
$\Gamma_{M}(x)=\frac{n}{x-t}$.
\end{lemma}
Let $G_{i}$ be a graph with $n_{i}$ vertices and $m_{i}$ edges. Let $V(G_{1})=\{v_{1},v_{2},\ldots,v_{n_{1}}\}$, $I(G_{1})=\{e_{1},e_{2},\ldots,e_{m_{1}}\}$, $V(G_{2})=\{u_{1},u_{2},\ldots,u_{n_{2}}\}$. 
For $i=1,2,\ldots,n_{1}$, let $V^{i}(G_{2})=\{u^{i}_{1},u^{i}_{2},\ldots,u^{i}_{n_{2}}\}$ be the vertex set of the $i^{th}$ copy of $G_{2}$. Then $V(G_{1})\bigcup I(G_{1})\bigcup \{V^{1}(G_{2})\bigcup V^{2}(G_{2})\bigcup \cdots \bigcup V^{l}(G_{2})\}$ is a partition of both $V(G_{1}\boxdot_{T} G_{2})$ and $V(G_{1}\boxminus_{T} G_{2})$, where $l=n_{1}$ for the former and $l=m_{1}$ for the latter.

\subsection{Spectra of $T$-vertex neighbourhood corona}
In this section we determine adjacency spectrum and Laplacian spectrum of $T$-vertex neighbourhood corona of two graphs.
\subsubsection{$A$-spectra of $T$-vertex neighbourhood corona}
Let $G_{1}$ be a $r_{1}$-regular graph on $n_{1}$ vertices and $m_{1}$ edges and $G_{2}$ be any arbitrary graph with $n_{2}$ vertices.
Then the adjacency matrix of $A(G_{1}\boxdot_{T} G_{2})$ can be written as:\\
$$
A(G_{1}\boxdot_{T} G_{2})=
\begin{pmatrix}
A(G_{1}) & R(G_{1}) & A(G_{1})\otimes \mathbf{1}_{n_{2}}^{T}\\\\
R(G_{1})^{T} & A(\mathcal{L}({G_{1}})) & R(G_{1})^{T}\otimes \mathbf{1}_{n_{2}}^{T}\\\\
A(G_{1})\otimes \mathbf{1}_{n_{2}} & R(G_{1})\otimes \mathbf{1}_{n_{2}} & I_{n_{1}}\otimes A(G_{2})
\end{pmatrix}.
$$
\begin{theorem}\label{t1}
Let $G_{1}$ be a $r_{1}$-regular graph on $n_{1}$ vertices and $m_{1}$ edges and $G_{2}$ be any arbitrary graph with $n_{2}$ vertices. Then the adjacency characteristic polynomial of $G_{1} \boxdot_{T} G_{2}$ be:
\begin{eqnarray*}
f_{A(G_{1} \boxdot_{T} G_{2})}(x)&=&
(x+2)^{m_{1}-n_{1}}\prod \limits_{i=1}^{n_{1}}\{(x+2)-(1+\Gamma_{A(G_{2})}(x))(\lambda_{i}(G_{1})+r_{1})\}\\&&
\det(xI_{n_{1}}-A(G_{1})-\Gamma_{A(G_{2}}(x)A(G_{1})^{2}\\&&-(1+\Gamma_{A(G_{2})}(x)A(G_{1}))R((x+2)I_{m_{1}}-(1+\Gamma_{A(G_{2})}(x))R^{T}R)^{-1}
R^{T}(1+\Gamma_{A(G_{2})}(x)A(G_{1}))).
\end{eqnarray*}
\end{theorem}
\begin{proof}
\small
The adjacency characteristic polynomial of $G_{1} \boxdot_{T} G_{2}$ is
\begin{eqnarray*}
f_{A(G_{1} \boxdot_{T} G_{2})}(x)
&=&\det(xI_{n_{1}(1+n_{2})+m_{1}}-A(G_{1}\boxdot_{T} G_{2}))\\
&=&\det\begin{pmatrix}
xI_{n_{1}}-A(G_{1}) & -R(G_{1}) & -A(G_{1})\otimes \mathbf{1}_{n_{2}}^{T}\\\\
-R(G_{1})^{T} & xI_{m_{1}}-A(\mathcal{L}({G_{1}})) & -R(G_{1})^{T}\otimes \mathbf{1}_{n_{2}}^{T}\\\\
-A(G_{1})\otimes \mathbf{1}_{n_{2}} & -R(G_{1})\otimes \mathbf{1}_{n_{2}} & I_{n_{1}}\otimes (xI_{n_{2}}-A(G_{2}))
\end{pmatrix}\\
&=&\det(I_{n_{1}}\otimes (xI_{n_{2}}-A(G_{2}))\det(S)=\prod \limits_{j=1}^{n_{2}}\{x-\lambda_{j}(G_{2})\}^{n_{1}}\det(S), where
\end{eqnarray*}
\begin{eqnarray*}
S&=&\begin{pmatrix}
xI_{n_{1}}-A(G_{1}) & -R(G_{1}) \\
-R(G_{1})^{T} & xI_{m_{1}}-A(\mathcal{L}({G_{1}}))
\end{pmatrix}
\\&&-
\begin{pmatrix}
-A(G_{1})\otimes \mathbf{1}_{n_{2}}^{T} \\
-R(G_{1})^{T}\otimes \mathbf{1}_{n_{2}}^{T}
\end{pmatrix}
(I_{n_{1}}\otimes (xI_{n_{2}}-A(G_{2})))^{-1}
\begin{pmatrix}
-A(G_{1})\otimes \mathbf{1}_{n_{2}} & -R(G_{1})\otimes \mathbf{1}_{n_{2}}
\end{pmatrix}\\\\
&=&\begin{pmatrix}
xI_{n_{1}}-A(G_{1})-\Gamma_{A(G_{2}}(x)A(G_{1})^{2}& -R-\Gamma_{A(G_{2})}(x)A(G_{1})R \\
-R^{T}-\Gamma_{A(G_{2})}(x)R^{T}A(G_{1})& xI_{m_{1}}-A(\mathcal{L}(G_{1}))-\Gamma_{A(G_{2})}(x)R^{T}R
\end{pmatrix}\\\\
&=&\begin{pmatrix}
xI_{n_{1}}-A(G_{1})-\Gamma_{A(G_{2}}(x)A(G_{1})^{2}& -R-\Gamma_{A(G_{2})}(x)A(G_{1})R \\
-R^{T}-\Gamma_{A(G_{2})}(x)R^{T}A(G_{1})& (x+2)I_{m_{1}}-(1+\Gamma_{A(G_{2})}(x))R^{T}R
\end{pmatrix}.
\end{eqnarray*}
\begin{eqnarray*}
\begin{array}{lcl}
\det(S)=\det((x+2)I_{m_{1}}-(1+\Gamma_{A(G_{2})}(x))R^{T}R) \det((xI_{n_{1}}-A(G_{1})-\Gamma_{A(G_{2})}(x)A(G_{1})^{2})\\\\~~~~~~~~~~~~-( R+\Gamma_{A(G_{2})}(x)A(G_{1})R)((x+2)I_{m_{1}}-(1+\Gamma_{A(G_{2})}(x))R^{T}R)^{-1}(R^{T}+\Gamma_{A(G_{2})}(x)R^{T}A(G_{1})))\\\\
~~~~~~~~=\det((x+2)I_{m_{1}}-(1+\Gamma_{A(G_{2})}(x))(A(\mathcal{L}(G_{1}))+2I_{m_{1}})) \det((xI_{n_{1}}-A(G_{1})-\Gamma_{A(G_{2})}(x)A(G_{1})^{2})\\\\~~~~~~~~~~~~-( R+\Gamma_{A(G_{2})}(x)A(G_{1})R)((x+2)I_{m_{1}}-(1+\Gamma_{A(G_{2})}(x))R^{T}R)^{-1}(R^{T}+\Gamma_{A(G_{2})}(x)R^{T}A(G_{1})))\\\\
~~~~~~~~=(x+2)^{m_{1}-n_{1}}\prod \limits_{i=1}^{n_{1}}\{(x+2)-(1+\Gamma_{A(G_{2})}(x))(\lambda_{i}(G_{1})+r_{1})\}\det((xI_{n_{1}}-A(G_{1})-\Gamma_{A(G_{2})}(x)A(G_{1})^{2})\\\\~~~~~~~~~~~~-( R+\Gamma_{A(G_{2})}(x)A(G_{1})R)((x+2)I_{m_{1}}-(1+\Gamma_{A(G_{2})}(x))R^{T}R)^{-1}(R^{T}+\Gamma_{A(G_{2})}(x)R^{T}A(G_{1})))
\end{array}
\end{eqnarray*}
Therefore
\begin{eqnarray*}
\begin{array}{lcl}
f_{A(G_{1} \boxdot_{T} G_{2})}(x)=\prod \limits_{j=1}^{n_{2}}\{x-\lambda_{j}(G_{2})\}^{n_{1}}(x+2)^{m_{1}-n_{1}}\prod \limits_{i=1}^{n_{1}}\{(x+2)-(1+\Gamma_{A(G_{2})}(x))(\lambda_{i}(G_{1})+r_{1})\}\\\\~~~~~~~~~~~~~~~~~~~~~~\det((xI_{n_{1}}-A(G_{1})-\Gamma_{A(G_{2})}(x)A(G_{1})^{2})-\\\\~~~~~~~~~~~~~~~~~~~~~~( 1+\Gamma_{A(G_{2})}(x)A(G_{1}))R((x+2)I_{m_{1}}-(1+\Gamma_{A(G_{2})}(x))R^{T}R)^{-1}R^{T}(1+\Gamma_{A(G_{2})}(x)A(G_{1}))).
\end{array}
\end{eqnarray*}
\end{proof}
\subsubsection{$L$-spectra of $T$-vertex neighbourhood corona}
Let $G_{1}$ be a $r_{1}$-regular graph on $n_{1}$ vertices and $m_{1}$ edges and $G_{2}$ be any arbitrary graph with $n_{2}$ vertices.
Then the adjacency matrix of $L(G_{1}\boxdot_{T} G_{2})$ can be written as:\\
$$
L(G_{1}\boxdot_{T} G_{2})=
\begin{pmatrix}
L(G_{1})+r_{1}(1+n_{2})I_{n_{1}} & -R(G_{1}) & -A(G_{1})\otimes \mathbf{1}_{n_{2}}^{T}\\\\
-R(G_{1})^{T} & (2n_{2}+2r_{1})I_{m_{1}}-A(\mathcal{L}({G_{1}})) & -R(G_{1})^{T}\otimes \mathbf{1}_{n_{2}}^{T}\\\\
-A(G_{1})\otimes \mathbf{1}_{n_{2}} & -R(G_{1})\otimes \mathbf{1}_{n_{2}} & I_{n_{1}}\otimes (L(G_{2})+2r_{1}I_{n_{2}})
\end{pmatrix}.
$$
\begin{theorem}\label{t2}
Let $G_{1}$ be a $r_{1}$-regular graph on $n_{1}$ vertices and $m_{1}$ edges and $G_{2}$ be any arbitrary graph with $n_{2}$ vertices. Then the Laplacian characteristic polynomial of $G_{1} \boxdot_{T} G_{2}$ be:
\begin{eqnarray*}
f_{L(G_{1} \boxdot_{T} G_{2})}(x)&=&
(x-2-2n_{2}-2r_{1})^{m_{1}-n_{1}}\prod \limits_{j=2}^{n_{2}}\{(x-2r_{1}-\mu_{i}(G_{2}))^{n_{1}}\}\\&&\prod \limits_{i=1}^{n_{1}}\{(x^{2}-(2+2n_{2}+2r_{1}+\mu_{i}(G_{1}))x+2r_{1}(2+2n_{2}+2r_{1})+(2r_{1}+n_{2})(\mu_{i}(G_{1})-2r_{1})\}\\&&
\det(((x-r_{1}(1+n_{2}))I_{n_{1}}-L(G_{1})-\Gamma_{L(G_{2})}(x-2r_{1})A(G_{1})^{2})-(R-\Gamma_{L(G_{2})}(x-2r_{1})A(G_{1})R)((x-2-2n_{2}-2r_{1})I_{m_{1}}+(1-\Gamma_{L(G_{2})}(x-2r_{1}))R^{T}R)^{-1}
(R^{T}+\Gamma_{L(G_{2})}(x-2r_{1})R^{T}A(G_{1}))).
\end{eqnarray*}
\end{theorem}
\begin{proof}
\small
The Laplacian characteristic polynomial of $G_{1} \boxdot_{T} G_{2}$ is\\\\
$f_{L(G_{1} \boxdot_{T} G_{2})}(x)$\\\\
=$\det(xI_{n_{1}(1+n_{2})+m_{1}}-L(G_{1}\boxdot_{T} G_{2}))$\\\\
=$\det\begin{pmatrix}
(x-r_{1}(1+n_{2})I_{n_{1}}-L(G_{1}) & R(G_{1}) & A(G_{1})\otimes \mathbf{1}_{n_{2}}^{T}\\\\
R(G_{1})^{T} & (x-2n_{2}-2r_{1})I_{m_{1}}+A(\mathcal{L}({G_{1}})) & R(G_{1})^{T}\otimes \mathbf{1}_{n_{2}}^{T}\\\\
A(G_{1})\otimes \mathbf{1}_{n_{2}} & R(G_{1})\otimes \mathbf{1}_{n_{2}} & I_{n_{1}}\otimes ((x-2r_{1})I_{n_{2}}-L(G_{2}))
\end{pmatrix}$\\\\
=$\det(I_{n_{1}}\otimes ((x-2r_{1})I_{n_{2}}-L(G_{2}))\det(S)$=$\prod \limits_{j=1}^{n_{2}}\{x-2r_{1}-\mu_{j}(G_{2})\}^{n_{1}}\det(S)$, where
\begin{eqnarray*}
S&=&\begin{pmatrix}
(x-r_{1}(1+n_{2}))I_{n_{1}}-L(G_{1}) & R(G_{1}) \\
R(G_{1})^{T} & (x-2n_{2}-2r_{1})I_{m_{1}}+A(\mathcal{L}({G_{1}}))
\end{pmatrix}
\\&&-
\begin{pmatrix}
A(G_{1})\otimes \mathbf{1}_{n_{2}}^{T} \\
R(G_{1})^{T}\otimes \mathbf{1}_{n_{2}}^{T}
\end{pmatrix}
(I_{n_{1}}\otimes ((x-2r_{1})I_{n_{2}}-L(G_{2})))^{-1}
\begin{pmatrix}
A(G_{1})\otimes \mathbf{1}_{n_{2}} & R(G_{1})\otimes \mathbf{1}_{n_{2}}
\end{pmatrix}\\\\
&=&\left(\begin{smallmatrix}
(x-r_{1}(1+n_{2}))I_{n_{1}}-L(G_{1})-\Gamma_{L(G_{2})}(x-2r_{1})A(G_{1})^{2} & R(G_{1})-\Gamma_{L(G_{2})}(x-2r_{1})A(G_{1})R(G_{1}) \\\\
R(G_{1})^{T}-\Gamma_{L(G_{2})}(x-2r_{1})R(G_{1})^{T}A(G_{1}) & (x-2n_{2}-2r_{1})I_{m_{1}}+A(\mathcal{L}({G_{1}}))-\Gamma_{L(G_{2})}(x-2r_{1})R(G_{1})^{T}R(G_{1})
\end{smallmatrix}\right)\\\\
&=&\left(\begin{smallmatrix}
(x-r_{1}(1+n_{2}))I_{n_{1}}-L(G_{1})-\Gamma_{L(G_{2})}(x-2r_{1})A(G_{1})^{2} & R(G_{1})-\Gamma_{L(G_{2})}(x-2r_{1})A(G_{1})R(G_{1}) \\\\
R(G_{1})^{T}-\Gamma_{L(G_{2})}(x-2r_{1})R(G_{1})^{T}A(G_{1}) & (x-2-2n_{2}-2r_{1})I_{m_{1}}+(1-\Gamma_{L(G_{2})}(x-2r_{1}))R(G_{1})^{T}R(G_{1})
\end{smallmatrix}\right).
\end{eqnarray*}
\begin{eqnarray*}
\begin{array}{lcl}
\det(S)=\det((x-2-2n_{2}-2r_{1})I_{m_{1}}+(1-\Gamma_{L(G_{2})}(x-2r_{1}))R(G_{1})^{T}R(G_{1}))\\~~~~~~~~~~~~~ \det(((x-r_{1}(1+n_{2}))I_{n_{1}}-L(G_{1})-\Gamma_{L(G_{2})}(x-2r_{1})A(G_{1})^{2})\\~~~~~~~~~~~~~~~-( R-\Gamma_{L(G_{2})}(x-2r_{1})A(G_{1})R)((x-2-2n_{2}-2r_{1})I_{m_{1}}+(1-\Gamma_{L(G_{2})}(x-2r_{1}))R(G_{1})^{T}R(G_{1}))^{-1}\\~~~~~~~~~~~~~~~~~~~~~(R^{T}-\Gamma_{L(G_{2})}(x-2r_{1})R^{T}A(G_{1})))\\\\
~~~~~~~~=\det((x-2-2n_{2}-2r_{1})I_{m_{1}}+(1-\Gamma_{L(G_{2})}(x-2r_{1}))(A(\mathcal{L}({G_{1}}))+2I_{m_{1}}))\\~~~~~~~~~~~~~ \det(((x-r_{1}(1+n_{2}))I_{n_{1}}-L(G_{1})-\Gamma_{L(G_{2})}(x-2r_{1})A(G_{1})^{2})\\~~~~~~~~~~~~~~~-( R-\Gamma_{L(G_{2})}(x-2r_{1})A(G_{1})R)((x-2-2n_{2}-2r_{1})I_{m_{1}}+(1-\Gamma_{L(G_{2})}(x-2r_{1}))R(G_{1})^{T}R(G_{1}))^{-1}\\~~~~~~~~~~~~~~~~~~~~~(R^{T}-\Gamma_{L(G_{2})}(x-2r_{1})R^{T}A(G_{1})))\\\\
~~~~~~~~=(x-2-2n_{2}-2r_{1})^{m_{1}-n_{1}}\prod \limits_{i=1}^{n_{1}}\{(x-2-2n_{2}-2r_{1})+(1-\Gamma_{L(G_{2})}(x-2r_{1}))(\lambda_{i}(G_{1})+r_{1})\}\\~~~~~~~~~~~~~ \det(((x-r_{1}(1+n_{2}))I_{n_{1}}-L(G_{1})-\Gamma_{L(G_{2})}(x-2r_{1})A(G_{1})^{2})\\~~~~~~~~~~~~~~~-( R-\Gamma_{L(G_{2})}(x-2r_{1})A(G_{1})R)((x-2-2n_{2}-2r_{1})I_{m_{1}}+(1-\Gamma_{L(G_{2})}(x-2r_{1}))R(G_{1})^{T}R(G_{1}))^{-1}\\~~~~~~~~~~~~~~~~~~~~~(R^{T}-\Gamma_{L(G_{2})}(x-2r_{1})R^{T}A(G_{1})))\\\\
~~~~~~~~=(x-2-2n_{2}-2r_{1})^{m_{1}-n_{1}}\prod \limits_{i=1}^{n_{1}}\{(x-2-2n_{2}-2r_{1})+(1-\frac{n_{2}}{x-2r_{1}})(\lambda_{i}(G_{1})+r_{1})\}\\~~~~~~~~~~~~~ \det(((x-r_{1}(1+n_{2}))I_{n_{1}}-L(G_{1})-\Gamma_{L(G_{2})}(x-2r_{1})A(G_{1})^{2})\\~~~~~~~~~~~~~~~-( R-\Gamma_{L(G_{2})}(x-2r_{1})A(G_{1})R)((x-2-2n_{2}-2r_{1})I_{m_{1}}+(1-\Gamma_{L(G_{2})}(x-2r_{1}))R(G_{1})^{T}R(G_{1}))^{-1}\\~~~~~~~~~~~~~~~~~~~~~(R^{T}-\Gamma_{L(G_{2})}(x-2r_{1})R^{T}A(G_{1})))\\\\
~~~~~~~~=\frac{(x-2-2n_{2}-2r_{1})^{m_{1}-n_{1}}}{(x-2r_{1})^{n_{1}}}\prod \limits_{i=1}^{n_{1}}\{(x-2-2n_{2}-2r_{1})(x-2r_{1})+(x-2r_{1}-n_{2})(\lambda_{i}(G_{1})+r_{1})\}\\~~~~~~~~~~~~~ \det(((x-r_{1}(1+n_{2}))I_{n_{1}}-L(G_{1})-\Gamma_{L(G_{2})}(x-2r_{1})A(G_{1})^{2})\\~~~~~~~~~~~~~~~-( R-\Gamma_{L(G_{2})}(x-2r_{1})A(G_{1})R)((x-2-2n_{2}-2r_{1})I_{m_{1}}+(1-\Gamma_{L(G_{2})}(x-2r_{1}))R(G_{1})^{T}R(G_{1}))^{-1}\\~~~~~~~~~~~~~~~~~~~~~(R^{T}-\Gamma_{L(G_{2})}(x-2r_{1})R^{T}A(G_{1})))\\\\
~~~~~~~~=\frac{(x-2-2n_{2}-2r_{1})^{m_{1}-n_{1}}}{(x-2r_{1})^{n_{1}}}\prod \limits_{i=1}^{n_{1}}\{x^{2}-(2+2r_{1}+2n_{2}+\mu_{i}(G_{1}))x+2r_{1}(2+2n_{2}+2r_{1})+(2r_{1}+n_{2})(\mu_{i}(G_{1})-2r_{1})\}\\~~~~~~~~~~~~~ \det(((x-r_{1}(1+n_{2}))I_{n_{1}}-L(G_{1})-\Gamma_{L(G_{2})}(x-2r_{1})A(G_{1})^{2})\\~~~~~~~~~~~~~~~-( R-\Gamma_{L(G_{2})}(x-2r_{1})A(G_{1})R)((x-2-2n_{2}-2r_{1})I_{m_{1}}+(1-\Gamma_{L(G_{2})}(x-2r_{1}))R(G_{1})^{T}R(G_{1}))^{-1}\\~~~~~~~~~~~~~~~~~~~~~(R^{T}-\Gamma_{L(G_{2})}(x-2r_{1})R^{T}A(G_{1})))\\\\
\end{array}
\end{eqnarray*}
Therefore
\begin{eqnarray*}
\begin{array}{lcl}
f_{L(G_{1} \boxdot_{T} G_{2})}(x)=(x-2-2n_{2}-2r_{1})^{m_{1}-n_{1}}\prod \limits_{j=2}^{n_{2}}\{(x-2r_{1}-\mu_{i}(G_{2}))^{n_{1}}\}\\~~~~\prod \limits_{i=1}^{n_{1}}\{(x^{2}-(2+2n_{2}+2r_{1}+\mu_{i}(G_{1}))x+2r_{1}(2+2n_{2}+2r_{1})+(2r_{1}+n_{2})(\mu_{i}(G_{1})-2r_{1})\}\\
\det(((x-r_{1}(1+n_{2}))I_{n_{1}}-L(G_{1})-\Gamma_{L(G_{2})}(x-2r_{1})A(G_{1})^{2})-(R-\Gamma_{L(G_{2})}(x-2r_{1})A(G_{1})R)((x-4-2n_{2}-r_{1})I_{m_{1}}+(1-\Gamma_{L(G_{2})}(x-2r_{1}))R^{T}R)^{-1}
(R^{T}+\Gamma_{L(G_{2})}(x-2r_{1})R^{T}A(G_{1}))).
\end{array}
\end{eqnarray*}
\end{proof}
\subsection{Spectra of $T$-edge neighbourhood corona}
In this section we determine adjacency spectrum and Laplacian spectrum of $T$-edge neighbourhood corona of two graphs.
\subsubsection{A-spectra of $T$-edge neighbourhood corona}
Let $G_{1}$ be a $r_{1}$-regular graph on $n_{1}$ vertices and $m_{1}$ edges and $G_{2}$ be any arbitrary graph with $n_{2}$ vertices.
Then the adjacency matrix of $A(G_{1}\boxminus_{T} G_{2})$ can be written as:\\
$$
A(G_{1}\boxminus_{T} G_{2})=
\begin{pmatrix}
A(G_{1}) & R(G_{1}) & R(G_{1})\otimes \mathbf{1}_{n_{2}}^{T}\\\\
R(G_{1})^{T} & A(\mathcal{L}({G_{1}})) & O_{m_{1}\times m_{1}n_{2}}\\\\
R(G_{1})^{T}\otimes \mathbf{1}_{n_{2}} & O_{m_{1}n_{2}\times m_{1}} & I_{m_{1}}\otimes A(G_{2})
\end{pmatrix}.
$$
\begin{theorem}\label{t3}
Let $G_{1}$ be a $r_{1}$-regular graph on $n_{1}$ vertices and $m_{1}$ edges and $G_{2}$ be any arbitrary graph with $n_{2}$ vertices. Then the adjacency characteristic polynomial of $G_{1} \boxminus_{T} G_{2}$ be:
\begin{eqnarray*}
f_{A(G_{1} \boxminus_{T} G_{2})}(x)&=&
(x+2)^{m_{1}-n_{1}}\prod \limits_{j=1}^{n_{2}}\{x-\lambda_{j}(G_{2})\}^{m_{1}}\\&&\prod \limits_{i=1}^{n_{1}}\{x^{2}+(2-r_{1}\Gamma_{A(G_{2})}(x)-r_{1}-(\Gamma_{A(G_{2})}(x)+2)\lambda_{i}(G_{1}))x+(1+\Gamma_{A(G_{2})}(x))(\lambda_{i}(G_{1}))^{2}\\&&+((2r_{1}-2)\Gamma_{A(G_{2})}(x)+r_{1}-3)\lambda_{i}(G_{1})+r_{1}(r_{1}-2)\Gamma_{A(G_{2})}(x)-r_{1})\}.
\end{eqnarray*}
\end{theorem}
\begin{proof}
\small
The adjacency characteristic polynomial of $G_{1} \boxminus_{T} G_{2}$ is
\begin{eqnarray*}
f_{A(G_{1} \boxminus_{T} G_{2})}(x)
&=&\det(xI_{m_{1}(1+n_{2})+n_{1}}-A(G_{1}\boxminus_{T} G_{2}))\\
&=&\det\begin{pmatrix}
xI_{n_{1}}-A(G_{1}) & -R(G_{1}) & -R(G_{1})\otimes \mathbf{1}_{n_{2}}^{T}\\\\
-R(G_{1})^{T} & xI_{m_{1}}-A(\mathcal{L}({G_{1}})) & O_{m_{1}\times m_{1}n_{2}}\\\\
-R(G_{1})^{T}\otimes \mathbf{1}_{n_{2}} & O_{m_{1}n_{2}\times m_{1}} & I_{m_{1}}\otimes (xI_{n_{2}}-A(G_{2}))
\end{pmatrix}\\
&=&\det(I_{m_{1}}\otimes (xI_{n_{2}}-A(G_{2}))\det(S)=\prod \limits_{j=1}^{n_{2}}\{x-\lambda_{j}(G_{2})\}^{n_{1}}\det(S), where
\end{eqnarray*}
\begin{eqnarray*}
S&=&\begin{pmatrix}
xI_{n_{1}}-A(G_{1}) & -R(G_{1}) \\
-R(G_{1})^{T} & xI_{m_{1}}-A(\mathcal{L}({G_{1}}))
\end{pmatrix}
\\&&-
\begin{pmatrix}
-R(G_{1})\otimes \mathbf{1}_{n_{2}}^{T} \\
O_{m_{1}\times m_{1}n_{2}}
\end{pmatrix}
(I_{m_{1}}\otimes (xI_{n_{2}}-A(G_{2})))^{-1}
\begin{pmatrix}
-R(G_{1})^{T}\otimes \mathbf{1}_{n_{2}} & O_{m_{1}n_{2}\times m_{1}}
\end{pmatrix}\\\\
&=&\begin{pmatrix}
xI_{n_{1}}-A(G_{1})-\Gamma_{A(G_{2})}(x)RR^{T} & -R \\
-R^{T} & xI_{m_{1}}-A(\mathcal{L}(G_{1})
\end{pmatrix}.
\end{eqnarray*}
\begin{eqnarray*}
\det(S)&=&\det\begin{pmatrix}
xI_{n_{1}}-A(G_{1})-\Gamma_{A(G_{2})}(x)RR^{T} & -R \\\\
-R^{T} & xI_{m_{1}}-A(\mathcal{L}(G_{1}))
\end{pmatrix}\\\\
&=&\det\begin{pmatrix}
(x+r_{1})I_{n_{1}}-(1+\Gamma_{A(G_{2})}(x))RR^{T} & -R \\\\
-R^{T} & xI_{m_{1}}-A(\mathcal{L}(G_{1}))
\end{pmatrix}\\\\
&=&\det\begin{pmatrix}
(x+r_{1})I_{n_{1}}-(1+\Gamma_{A(G_{2})}(x))RR^{T} & -R \\\\
-(1+x+r_{1})R^{T}+(1+\Gamma_{A(G_{2})}(x))R^{T}RR^{T} & (x+2)I_{m_{1}}
\end{pmatrix}\\\\
&=&\det\begin{pmatrix}
(x+r_{1})I_{n_{1}}-(1+\Gamma_{A(G_{2})}(x))RR^{T}-\frac{1+x+r_{1}}{x+2}RR^{T}+\frac{1+\Gamma_{A(G_{2})}(x)}{x+2}RR^{T}RR^{T} & O\\\\
-(1+x+r_{1})R^{T}+(1+\Gamma_{A(G_{2})}(x))R^{T}RR^{T} & (x+2)I_{m_{1}}.
\end{pmatrix}
\end{eqnarray*}
\begin{eqnarray*}
\begin{array}{lcl}
\det(S)=\det((x+2)I_{m_{1}}) \det((x+r_{1})I_{n_{1}}+\frac{(1+\Gamma_{A(G_{2})}(x))}{x+2}A(G_{1})^{2}+\frac{2r_{1}(1+\Gamma_{A(G_{2})}(x))-2x-r_{1}-3-(x+2)\Gamma_{A(G_{2})}(x)}{x+2}A(G_{1})\\~~~~~~~~~~~~~~~~~~~~~~~~~~~~~~~~~~~~~~~~+\frac{r_{1}^{2}(1+\Gamma_{A(G_{2})}(x))-2xr_{1}-r_{1}^{2}-3r_{1}-(x+2)r_{1}\Gamma_{A(G_{2})}(x)}{x+2}I_{n_{1}})\\\\
~~~~~~~~=(x+2)^{m_{1}} \det(\frac{(1+\Gamma_{A(G_{2})}(x))}{x+2}A(G_{1})^{2}+\frac{(-\Gamma_{A(G_{2})}(x)-2)x-r_{1}-3-(2r_{1}-2)\Gamma_{A(G_{2})}(x)}{x+2}A(G_{1})\\~~~~~~~~~~~~~~~~~~~~~~~~~~~~~~~~~~~~~~~~+\frac{x^{2}+(2-r_{1}\Gamma_{A(G_{2})}(x)-r_{1})x+(r_{1}^{2}\Gamma_{A(G_{2})}(x)-2r_{1}\Gamma_{A(G_{2})}(x)-r_{1})}{x+2}I_{n_{1}})\\\\
~~~~~~~~=(x+2)^{m_{1}-n_{1}}\prod \limits_{i=1}^{n_{1}}\{x^{2}+(2-r_{1}\Gamma_{A(G_{2})}(x)-r_{1}-(\Gamma_{A(G_{2})}(x)+2)\lambda_{i}(G_{1}))x+(1+\Gamma_{A(G_{2})}(x))(\lambda_{i}(G_{1}))^{2}\\~~~~~~~~~~~~~~~~~~~~~~~~~~~~~~~+
((2r_{1}-2)\Gamma_{A(G_{2})}(x)+r_{1}-3)\lambda_{i}(G_{1})+r_{1}(r_{1}-2)\Gamma_{A(G_{2})}(x)-r_{1})\}\\\\
\end{array}
\end{eqnarray*}
Therefore
\begin{eqnarray*}
\begin{array}{lcl}
f_{A(G_{1} \boxminus_{T} G_{2})}(x)=(x+2)^{m_{1}-n_{1}}\prod \limits_{j=1}^{n_{2}}\{x-\lambda_{j}(G_{2})\}^{m_{1}}\prod \limits_{i=1}^{n_{1}}\{x^{2}+(2-r_{1}\Gamma_{A(G_{2})}(x)-r_{1}-(\Gamma_{A(G_{2})}(x)+2)\lambda_{i}(G_{1}))x\\~~~~~~~~~~~~~~~~~~~~~~~+(1+\Gamma_{A(G_{2})}(x))(\lambda_{i}(G_{1}))^{2}+((2r_{1}-2)\Gamma_{A(G_{2})}(x)+r_{1}-3)\lambda_{i}(G_{1})+r_{1}(r_{1}-2)\Gamma_{A(G_{2})}(x)-r_{1})\}.
\end{array}
\end{eqnarray*}
\end{proof}

\begin{corollary}
For $i=1,2$, let $G_{i}$ be an $r_{i}$-regular graph with $n_{i}$ vertices and $m_{i}$ edges. Then the adjacency spectrum of $G_{1}\boxminus_{T} G_{2}$ consists of:
\begin{enumerate}[(i)]
\item The eigenvalue $\lambda_{j}(G_{2})$ with multiplicity $m_{1}$ for every eigenvalue $\lambda_{j}$ $(j=2,3,\ldots,n_{2})$ of $A(G_{2})$,
\item The eigenvalue $r_{2}$ with multiplicity $m_{1}-n_{1}$,
\item The eigenvalue $-2$ with multiplicity $m_{1}-n_{1}$,
\item Three roots of the equation\\
$x^{3}+(2-r_{1}-r_{2}-2\lambda_{i}(G_{1}))x^{2}+(r_{1}r_{2}-r_{1}n_{2}-3r_{1}-(n_{2}-3r_{1}+3)\lambda_{i}(G_{1})+\lambda_{i}(G_{1})^{2})x+(n_{2}-r_{2})(\lambda_{i}(G_{1}))^{2}+(2r_{1}n_{2}-2n_{2}-r_{1}r_{2}+3r_{2})\lambda_{i}(G_{1})+r_{1}(r_{1}-2)n_{2}+r_{1}r_{2}=0$,\\
for each eigenvalue $\lambda_{i}$ $(i=1,2,\ldots,n_{1})$ of $A(G_{1})$.
\end{enumerate}
\end{corollary}
\begin{corollary}
If $G_{1}$ be a $r_{1}$-regular graph on $n_{1}$ vertices and $m_{1}$ edges and $G_{2}$ be complete bipartite graph $K_{p,q}$, then  the adjacency spectrum of $G_{1}\boxminus_{T} K_{p,q}$ consists of:
\begin{enumerate}[(i)]
\item The eigenvalue $0$ with multiplicity $m_{1}(p+q-2)$,
\item The eigenvalue $pq$ with multiplicity $m_{1}-n_{1}$,
\item The eigenvalue $-2$ with multiplicity $m_{1}-n_{1}$,
\item Four roots of the equation\\
$x^{4}+(2-r_{1}-2\lambda_{i}(G_{1}))x^{3}+(-pq-(p+q)(r_{1}+\lambda_{i}(G_{1}))+(\lambda_{i}(G_{1}))^{2}+(r_{1}-3)\lambda_{i}(G_{1})-r_{1})x^{2}+(-2pq-r_{1}pq+(p+q)(\lambda_{i}(G_{1}))^{2}+(2r_{1}-2)(p+q)\lambda_{i}(G_{1})+r_{1}(r_{1}-2)(p+q))x+pq(\lambda_{i}(G_{1}))^{2}+((2r_{1}-2)2pq-pq(r_{1}-3))\lambda_{i}(G_{1})+r_{1}(r_{1}-2)2pq+r_{1}pq=0$,\\
for each eigenvalue $\lambda_{i}$ $(i=1,2,\ldots,n_{1})$ of $A(G_{1})$.
\end{enumerate}
\end{corollary}
\begin{corollary}
\begin{enumerate}[(a)]
\item If $H_{1}$ and $H_{2}$ are $A$-cospectral regular graphs, and $H$ is a regular graph, then $H_{1}\boxminus_{T} H$ and $H_{2}\boxminus_{T} H$; and $H\boxminus_{T} H_{1}$ and $H\boxminus_{T} H_{2}$ are $A$-cospectral.
\item If $F_{1}$ and $F_{2}$; and $H_{1}$ and $H_{2}$ are $A$-cospectral regular graphs, then $F_{1}\boxminus_{T} H_{1}$ and $F_{2}\boxminus_{T} H_{2}$ are $A$-cospectral.
\end{enumerate}
\end{corollary}
\subsubsection{$L$-spectra of $T$-edge neighbourhood corona}
Let $G_{1}$ be a $r_{1}$-regular graph on $n_{1}$ vertices and $m_{1}$ edges and $G_{2}$ be any arbitrary graph with $n_{2}$ vertices.
Then the Laplacian matrix of $G_{1}\boxminus_{T} G_{2}$ can be written as:\\
$$
L(G_{1}\boxminus_{T} G_{2})=
\begin{pmatrix}
L(G_{1})+r_{1}(1+n_{2})I_{n_{1}} & -R(G_{1}) & -R(G_{1})\otimes \mathbf{1}_{n_{2}}^{T}\\\\
-R(G_{1})^{T} & 2r_{1}I_{m_{1}}-A(\mathcal{L}({G_{1}})) & O_{m_{1}\times m_{1}n_{2}}\\\\
-R(G_{1})^{T}\otimes \mathbf{1}_{n_{2}} & O_{m_{1}n_{2}\times m_{1}} & I_{m_{1}}\otimes (L(G_{2})+2I_{n_{2}})
\end{pmatrix}.
$$
\begin{theorem}\label{t4}
Let $G_{1}$ be a $r_{1}$-regular graph on $n_{1}$ vertices and $m_{1}$ edges and $G_{2}$ be any arbitrary graph with $n_{2}$ vertices. Then the Laplacian characteristic polynomial of $G_{1} \boxminus_{T} G_{2}$ be:
\begin{eqnarray*}
f_{L(G_{1} \boxminus_{T} G_{2})}(x)&=&
(x-2-2r_{1})^{m_{1}-n_{1}}\prod \limits_{j=2}^{n_{2}}\{(x-2-\mu_{i}(G_{2}))^{m_{1}}\}\\&&\prod \limits_{i=1}^{n_{1}}\{(x^{2}-(r_{1}(7+n_{2})+2-r_{1}\Gamma_{L({G_{2}})}(x-2)+(2-\Gamma_{L({G_{2}})}(x-2))(r_{1}-\mu_{i}(G_{1})))x\\&&+(1+\Gamma_{L({G_{2}})}(x-2))(r_{1}-\mu_{i}(G_{1}))^{2}+r_{1}(3+n_{2})(2r_{1}+2)+4r_{1}^{2}+3r_{1}+r_{1}^{2}n_{2}\\&&+(r_{1}(7+n_{2})+3-(4r_{1}+2)\Gamma_{L({G_{2}})}(x-2))(r_{1}-\mu_{i}(G_{1}))-r_{1}(r_{1}+2)\Gamma_{L({G_{2}})}(x-2)).
\end{eqnarray*}
\end{theorem}
\begin{proof}
\small
The Laplacian characteristic polynomial of $G_{1} \boxminus_{T} G_{2}$ is
\begin{eqnarray*}
f_{L(G_{1} \boxminus_{T} G_{2})}(x)
&=&\det(xI_{m_{1}(1+n_{2})+n_{1}}-L(G_{1}\boxminus_{T} G_{2}))\\
&=&\det\begin{pmatrix}
(x-r_{1}(1+n_{2})I_{n_{1}}-L(G_{1}) & R(G_{1}) & R(G_{1})\otimes \mathbf{1}_{n_{2}}^{T}\\\\
R(G_{1})^{T} & (x-2r_{1})I_{m_{1}}+A(\mathcal{L}({G_{1}})) & O_{m_{1}\times m_{1}n_{2}}\\\\
R(G_{1})^{T}\otimes \mathbf{1}_{n_{2}} & O_{m_{1}n_{2}\times m_{1}} & I_{m_{1}}\otimes ((x-2)I_{n_{2}}-L(G_{2}))
\end{pmatrix}\\
&=&\det(I_{n_{1}}\otimes ((x-2)I_{n_{2}}-L(G_{2}))\det(S)=\prod \limits_{j=1}^{n_{2}}\{x-2-\mu_{j}(G_{2})\}^{m_{1}}\det(S), where
\end{eqnarray*}
\begin{eqnarray*}
S&=&\begin{pmatrix}
(x-r_{1}(1+n_{2}))I_{n_{1}}-L(G_{1}) & R(G_{1}) \\
R(G_{1})^{T} & (x-2r_{1})I_{m_{1}}+A(\mathcal{L}({G_{1}}))
\end{pmatrix}
\\&&-
\begin{pmatrix}
R(G_{1})\otimes \mathbf{1}_{n_{2}}^{T} \\
O_{m_{1}\times m_{1}n_{2}}
\end{pmatrix}
(I_{m_{1}}\otimes ((x-2)I_{n_{2}}-L(G_{2})))^{-1}
\begin{pmatrix}
R(G_{1})^{T}\otimes \mathbf{1}_{n_{2}} & O_{m_{1}n_{2}\times m_{1}}
\end{pmatrix}\\\\
&=&\left(\begin{smallmatrix}
(x-r_{1}(1+n_{2}))I_{n_{1}}-L(G_{1})-RR^{T}\Gamma_{L({G_{2}})}(x-2) & R(G_{1})\\\\
R(G_{1})^{T} & (x-2r_{1})I_{m_{1}}+A(\mathcal{L}({G_{1}}))
\end{smallmatrix}\right)\\\\
&=&\left(\begin{smallmatrix}
(x-r_{1}(2+n_{2}))I_{n_{1}}+A(G_{1})-RR^{T}\Gamma_{L({G_{2}})}(x-2) & R(G_{1})\\\\
R(G_{1})^{T} & (x-2r_{1}-2)I_{m_{1}}+R(G_{1})^{T}R(G_{1})
\end{smallmatrix}\right)\\\\
&=&\left(\begin{smallmatrix}
(x-r_{1}(3+n_{2}))I_{n_{1}}+(1-\Gamma_{L({G_{2}})}(x-2))RR^{T} & R(G_{1})\\\\
R(G_{1})^{T} & (x-2r_{1}-2)I_{m_{1}}+R(G_{1})^{T}R(G_{1})
\end{smallmatrix}\right)\\\\
&=&\left(\begin{smallmatrix}
(x-r_{1}(3+n_{2}))I_{n_{1}}+(1-\Gamma_{L({G_{2}})}(x-2))RR^{T} & R(G_{1})\\\\
(1-x+r_{1}(3+n_{2}))R(G_{1})^{T}-(1-\Gamma_{L({G_{2}})}(x-2))R^{T}RR^{T} & (x-2r_{1}-2)I_{m_{1}}
\end{smallmatrix}\right)\\\\
&=&\left(\begin{smallmatrix}
(x-r_{1}(3+n_{2}))I_{n_{1}}+(1-\Gamma_{L({G_{2}})}(x-2))RR^{T}-\frac{1-x+r_{1}(3+n_{2})}{x-2r_{1}-2}RR^{T}+\frac{1-\Gamma_{L({G_{2}})}(x-2)}{(x-2r_{1}-2)} RR^{T}RR^{T}& O\\\\
(1-x+r_{1}(3+n_{2}))R(G_{1})^{T}-(1-\Gamma_{L({G_{2}})}(x-2))R^{T}RR^{T} & (x-2r_{1}-2)I_{m_{1}}
\end{smallmatrix}\right)
\end{eqnarray*}
\begin{eqnarray*}
\begin{array}{lcl}
\det(S)=(x-2r_{1}-2)^{m_{1}} \det(((x-r_{1}(3+n_{2}))I_{n_{1}}-\frac{x-2r_{1}-2-(x-2r_{1}-2)\Gamma_{L({G_{2}})}(x-2)-1+x-r_{1}(3+n_{2})}{x-2r_{1}-2}(A(G_{1})+r_{1}I_{n_{1}})\\~~~~~~~~~~~~~~~+\frac{1-\Gamma_{L({G_{2}})}(x-2)}{x-2r_{1}-2}(A(G_{1})+r_{1}I_{n_{1}})^{2})\\\\
~~~~~~~~=(x-2r_{1}-2)^{m_{1}} \det(((x-r_{1}(3+n_{2}))I_{n_{1}}-\frac{2x-5r_{1}-3-(x-2r_{1}-2)\Gamma_{L({G_{2}})}(x-2)-r_{1}n_{2}}{x-2r_{1}-2}(A(G_{1})+r_{1}I_{n_{1}})\\~~~~~~~~~~~~~~~+\frac{1-\Gamma_{L({G_{2}})}(x-2)}{x-2r_{1}-2}(A(G_{1})+r_{1}I_{n_{1}})^{2})\\\\
~~~~~~~~=(x-2r_{1}-2)^{m_{1}} \det(((x-r_{1}(3+n_{2}))I_{n_{1}}-(\frac{2x-5r_{1}-3-(x-2r_{1}-2)\Gamma_{L({G_{2}})}(x-2)-r_{1}n_{2}}{x-2r_{1}-2}-\frac{2r_{1}(1-\Gamma_{L({G_{2}})}(x-2))}{x-2r_{1}-2})A(G_{1})\\~~~~~~~~~~~~~~~+\frac{r_{1}(2x-5r_{1}-3-(x-2r_{1}-2)\Gamma_{L({G_{2}})}(x-2)-r_{1}n_{2})+r_{1}^{2}-r_{1}^{2}\Gamma_{L({G_{2}})}(x-2)}{x-2r_{1}-2}I_{n_{1}}+\frac{1-\Gamma_{L({G_{2}})}(x-2)}{x-2r_{1}-2}A(G_{1})^{2})\\\\
~~~~~~~~=(x-2-2r_{1})^{m_{1}-n_{1}}\prod \limits_{i=1}^{n_{1}}\{x^{2}-(3r_{1}+r_{1}n_{2}+2r_{1}+2+2r_{1}-r_{1}\Gamma_{L({G_{2}})}(x-2)+(2-\Gamma_{L({G_{2}})}(x-2))\lambda_{i}(G_{1}))x\\~~~~~~~~~~~+(1-\Gamma_{L({G_{2}})}(x-2))(\lambda_{i}(G_{1}))^{2}+r_{1}(3+n_{2})(2r_{1}+2)+4r_{1}^{2}+3r_{1}+r_{1}^{2}n_{2}-r_{1}(r_{1}+2)\Gamma_{L({G_{2}})}(x-2)\\~~~~~~~~~~~~~~+(5r_{1}+3+r_{1}n_{2}-(2r_{1}+2)\Gamma_{L({G_{2}})}(x-2)+2r_{1}(1-\Gamma_{L({G_{2}})}(x-2)))\lambda_{i}(G_{1})\}\\\\
~~~~~~~~=(x-2-2r_{1})^{m_{1}-n_{1}}\prod \limits_{i=1}^{n_{1}}\{(x^{2}-(r_{1}(7+n_{2})+2-r_{1}\Gamma_{L({G_{2}})}(x-2)+(2-\Gamma_{L({G_{2}})}(x-2))(r_{1}-\mu_{i}(G_{1})))x\\~~~~~~~~~~~+(1+\Gamma_{L({G_{2}})}(x-2))(r_{1}-\mu_{i}(G_{1}))^{2}+(r_{1}(7+n_{2})+3-(4r_{1}+2)\Gamma_{L({G_{2}})}(x-2))(r_{1}-\mu_{i}(G_{1}))\\~~~~~~~~~~~~+r_{1}(3+n_{2})(2r_{1}+2)+4r_{1}^{2}+3r_{1}+r_{1}^{2}n_{2}-r_{1}(r_{1}+2)\Gamma_{L({G_{2}})}(x-2))\\\\
\end{array}
\end{eqnarray*}
Therefore
\begin{eqnarray*}
\begin{array}{lcl}
f_{L(G_{1} \boxminus_{T} G_{2})}(x)=(x-2-2r_{1})^{m_{1}-n_{1}}\prod \limits_{j=2}^{n_{2}}\{(x-2-\mu_{i}(G_{2}))^{m_{1}}\}\\\prod \limits_{i=1}^{n_{1}}\{(x^{2}-(r_{1}(7+n_{2})+2-r_{1}\Gamma_{L({G_{2}})}(x-2)+(2-\Gamma_{L({G_{2}})}(x-2))(r_{1}-\mu_{i}(G_{1})))x\\~~~~~~~~+(1+\Gamma_{L({G_{2}})}(x-2))(r_{1}-\mu_{i}(G_{1}))^{2}+(r_{1}(7+n_{2})+3-(4r_{1}+2)\Gamma_{L({G_{2}})}(x-2))(r_{1}-\mu_{i}(G_{1}))\\~~~~~~~~~+r_{1}(3+n_{2})(2r_{1}+2)+4r_{1}^{2}+3r_{1}+r_{1}^{2}n_{2}-r_{1}(r_{1}+2)\Gamma_{L({G_{2}})}(x-2)\}
\end{array}
\end{eqnarray*}
\end{proof}
\begin{corollary}
For $i=1,2$, let $G_{i}$ be an $r_{i}$-regular graph with $n_{i}$ vertices and $m_{i}$ edges. Then the Laplacian spectrum of $G_{1}\boxminus_{T} G_{2}$ consists of:
\begin{enumerate}[(i)]
\item The eigenvalue $2+\mu_{j}(G_{2})$ with multiplicity $m_{1}$ for every eigenvalue $\mu_{j}$ $(j=2,3,\ldots,n_{2})$ of $L(G_{2})$,
\item The eigenvalue $2+2r_{1}$ with multiplicity $m_{1}-n_{1}$,
\item The eigenvalue $2$ with multiplicity $m_{1}-n_{1}$,
\item Three roots of the equation\\
$x^{3}-(r_{1}(7+n_{2})+4+2(r_{1}-\mu_{i}(G_{1})))x^{2}+(2r_{1}(7+n_{2})+4+r_{1}n_{2}+(7r_{1}+r_{1}n_{2}+7+n_{2})(r_{1}-\mu_{i}(G_{1}))+(r_{1}-\mu_{i}(G_{1}))^{2}+r_{1}(3+n_{2})(2r_{1}+2)+4r_{1}^{2}+3r_{1}+r_{1}^{2}n_{2})x-(2+n_{2})(r_{1}-\mu_{i}(G_{1}))^{2}-(2r_{1}(7+n_{2})+6+4r_{1}n_{2}+2n_{2})(r_{1}-\mu_{i}(G_{1}))-2r_{1}(3+n_{2})(2r_{1}+2)-2(4r_{1}^{2}+3r_{1}+r_{1}^{2}n_{2})-r_{1}(r_{1}+2)n_{2}=0$,\\
for each eigenvalue $\mu_{i}$ $(i=1,2,\ldots,n_{1})$ of $L(G_{1})$.
\end{enumerate}
\end{corollary}
\begin{corollary}
\begin{enumerate}[(a)]
\item If $H_{1}$ and $H_{2}$ are $L$-cospectral regular graphs, and $H$ is a regular graph, then $H_{1}\boxminus_{T} H$ and $H_{2}\boxminus_{T} H$; and $H\boxminus_{T} H_{1}$ and $H\boxminus_{T} H_{2}$ are $L$-cospectral.
\item If $F_{1}$ and $F_{2}$; and $H_{1}$ and $H_{2}$ are $L$-cospectral regular graphs, then $F_{1}\boxminus_{T} H_{1}$ and $F_{2}\boxminus_{T} H_{2}$ are $L$-cospectral.
\end{enumerate}
\end{corollary}

\end{document}